\documentclass{article}
\usepackage{graphicx} 

\usepackage{amsopn}
\usepackage{amsmath,amsthm,amssymb}

\usepackage[T1]{fontenc}
\usepackage{tabularx}
\usepackage[margin=2cm]{geometry}
\usepackage{amsfonts}
\usepackage[english]{babel}
\usepackage[utf8]{inputenc}
\usepackage{hyperref}
\hypersetup{colorlinks=true, linkcolor=black, citecolor=black, urlcolor=black, hypertexnames=false}
\usepackage{cleveref}

\usepackage{xcolor}
\usepackage{adjustbox}
\usepackage{nicematrix}
\usepackage{tikz}
\usepackage{float}

\newcommand{\g}{\mathfrak{g}}
\newcommand{\h}{\mathfrak{h}}
\renewcommand{\u}{\mathfrak{u}}

\newcommand{\R}{\mathbb{R}}
\newcommand{\<}{\langle}
\renewcommand{\>}{\rangle}

\newcommand{\Z}{\xi}

\newcommand{\nc}{\newcommand}
\nc{\PP}{{\mathbb P}}
\nc{\aff}{\mathfrak{aff} }
\nc{\ad}{\operatorname{ad}}
\nc{\tr}{\operatorname{tr}}
\nc{\alt}{\raise1pt\hbox{$\bigwedge$}}
\nc{\prodint}{\langle \cdotp,\cdotp \rangle }

\newcommand{\ri}{{\rm (i)}}
\newcommand{\rii}{{\rm (ii)}}
\newcommand{\riii}{{\rm (iii)}}

\theoremstyle{plain}
\newtheorem{thm}{Theorem}[section]
\newtheorem{prop}[thm]{Proposition}
\newtheorem{teo}[thm]{Theorem}
\newtheorem{coro}[thm]{Corollary}
\newtheorem{lema}[thm]{Lemma}

\theoremstyle{definition}
\newtheorem{de}[thm]{Definition}

\newtheorem{obs}[thm]{Observation}



\title{On coKähler structures on Lie algebras}
\author{Javier Liendo\thanks{Universidad Nacional de C\'ordoba, 5000 C\'ordoba, Argentina. Email: \texttt{javierliendo@unc.edu.ar}.}
	\and
	Marcos Origlia\thanks{Universidad Nacional de C\'ordoba, 5000 C\'ordoba, Argentina; Guangdong Technion -- Israel Institute of Technology, Shantou, China. Email: \texttt{marcos.origlia@unc.edu.ar}.}}
\date{}

\begin{document}
	
	\maketitle
	\begin{abstract}
		We study coKähler structures on real Lie algebras by using the Fino--Vezzoni correspondence, which relates them to Kähler Lie algebras endowed with compatible skew-symmetric derivations. 
		Our first result completes the flat case: every odd-dimensional flat Lie algebra admits a coKähler structure, giving a converse to the theorem of Fino and Vezzoni asserting that unimodular coKähler Lie algebras are flat and hence solvable. 
		We then characterize almost abelian Lie algebras carrying coKähler structures and we determine the possible Reeb directions. 
		As a consequence, every non-unimodular almost abelian coKähler Lie algebra splits as the direct product of a non-unimodular Kähler Lie algebra and a line. 
		Finally, we apply these results in low dimensions: we classify the almost abelian cases in dimensions five and seven, and we give the classification, up to Lie algebra isomorphism, of five-dimensional coKähler Lie algebras by reducing the corresponding extensions arising from four-dimensional Kähler Lie algebras. 
		The resulting examples show that the direct product splitting above is a special feature of the almost abelian setting.
	\end{abstract}

	\section{Introduction}
	CoKähler geometry is the odd-dimensional counterpart of Kähler geometry (see, for instance, \cite{Li}). 
	For Lie groups endowed with left-invariant structures, the problem can be formulated entirely at the Lie algebra level. 
	A theorem of Fino and Vezzoni identifies coKähler Lie algebras with Kähler Lie algebras equipped with a skew-symmetric derivation commuting with the complex structure. 
	Thus, in the left-invariant setting, the odd-dimensional theory is controlled by Kähler Lie algebras together with one additional algebraic datum. 
	This correspondence is the main tool of the paper: we use it both to prove structural results and to obtain explicit classifications for almost abelian Lie algebras.
	
	Fino and Vezzoni proved that every unimodular Lie algebra admitting a coKähler structure is flat and hence solvable. 
	This gives a strong obstruction, and it naturally leads to the converse problem in the flat setting. 
	We prove that every odd-dimensional flat Lie algebra admits a coKähler structure. 
	The proof relies on the description of flat Lie algebras due to Barberis, Dotti and Fino, together with the fact that every even-dimensional flat Lie algebra admits a Kähler structure. 
	Consequently, the flat part of the theory is complete: unimodular coKähler Lie algebras are flat, and every odd-dimensional flat Lie algebra carries a coKähler structure.
	
	We then study the coKähler condition in the almost abelian setting. 
	Recall that an almost abelian Lie algebra is a Lie algebra with an abelian ideal of codimension one; equivalently, it can be written as $\g=\R x\ltimes_M\u$, where the bracket is encoded by a single endomorphism $M$ of $\u$. 
	We give a characterization of those matrices $M$ for which $\g$ admits a coKähler structure. 
	In one case, $M$ is unitary with respect to a suitable Hermitian structure. 
	In the remaining case, after choosing an adapted basis, $M$ has one distinguished real direction and a unitary block on the complementary invariant subspace. 
	The same argument also describes all possible Reeb directions.
	
	A consequence of this characterization is a splitting result. 
	If an almost abelian coKähler Lie algebra is non-unimodular, then it is not only obtained from a Kähler Lie algebra by adding a skew-symmetric derivation: it actually splits as a direct product $\h\times\R$, where $\h$ is a non-unimodular almost abelian Kähler Lie algebra. 
	This shows that the almost abelian case is particularly rigid, and it raises the question of whether such a product decomposition should hold more generally for non-unimodular coKähler Lie algebras.
	
	The final part of the paper addresses this question in low dimension and connects the preceding results with known classifications. 
	Using the almost abelian characterization, we first list the almost abelian coKähler Lie algebras in dimensions five and seven. 
	We then classify all five-dimensional coKähler Lie algebras up to Lie algebra isomorphism. 
	By the Fino--Vezzoni correspondence, these Lie algebras arise from four-dimensional Kähler Lie algebras and compatible skew-symmetric derivations; the relevant four-dimensional Kähler Lie algebras are taken from Ovando's classification. 
	Calvaruso and Perrone classified five-dimensional coKähler structures up to equivalence of the full structure. 
	Here we focus instead on the underlying Lie algebras, and several of the corresponding extensions become isomorphic. 
	After identifying these repetitions, we obtain a reduced five-dimensional list. 
	The non-unimodular examples in this list show that the product splitting proved for almost abelian Lie algebras does not hold in general.
	
	\section{Preliminaries}
	\subsection{CoKähler structures}
	We recall the geometric structures studied in this paper, first on manifolds and then in the left-invariant setting on Lie algebras.
	
	\begin{de}
		Let $M$ be a $(2n+1)$-dimensional smooth manifold. An {\em almost contact metric structure} on $M$ is a quadruple $(J,\Z,\alpha,g)$, where $J$ is a smooth $(1,1)$-tensor, $\Z$ is a smooth vector field, $\alpha$ is a smooth $1$-form and $g$ is a Riemannian metric on $M$ satisfying the following conditions:
		\[J^2=-\mathrm{Id}+\alpha\otimes \Z, \quad \alpha(\Z)=1,\]
		\[g(JX,JY)=g(X,Y)-\alpha(X)\alpha(Y), \quad X,Y\in\Gamma(TM).\]
		An almost contact metric structure $(J,\Z,\alpha,g)$ is said to be {\em normal} if
		\[N_J(X,Y)=-2d\alpha(X,Y)\Z, \quad \text{for any $X,Y\in \Gamma(TM)$},\]
		where $N_J$ is the Nijenhuis tensor defined by 
		\[N_J(X,Y)=[JX,JY]-J[JX,Y]-J[X,JY]+J^2[X,Y].\]
	\end{de}
	
	Let $(M,J,\Z,\alpha,g)$ be an almost contact metric manifold. It follows that $|\Z|= 1$, $J\Z = 0$, $\alpha\circ J = 0$, and $J$ is skew-symmetric. We define the {\em fundamental form} on $M$ as the $2$-form given by 
	\[\omega(\cdot,\cdot):=g(J\cdot,\cdot).\]
	\begin{de}
		A normal almost contact metric structure $(J,\Z,\alpha,g)$ is said to be a {\em coKähler structure} if 
		\[d\alpha=0, \quad d\omega=0.\]
	\end{de}
	CoK\"ahler structures are also known as cosymplectic structures; see, for instance, \cite{Blair}. The terminology coK\"ahler emphasizes their analogy with K\"ahler geometry.
	
	We focus on left-invariant coKähler structures on odd-dimensional Lie groups.
	Recall that a {\em coKähler structure} on a real Lie algebra $\g$ of dimension $2n+1$ is a quadruple $(J,\Z,\alpha,\prodint)$, where $\Z$ is a non-zero vector of $\g$, $\alpha\in\g^*$, $J$ is an endomorphism of $\g$, and $\prodint$ is an inner product on $\g$ satisfying
\[J^2=-\mathrm{Id}+\alpha \otimes \Z, \quad d\alpha=0, \quad N_J=0,\]
\[\<J\cdot,J\cdot\>=\prodint- \alpha \otimes \alpha(\cdot, \cdot), \quad d\omega=0,\]
where $\omega$ is the fundamental form.
	
If $G$ is the simply connected Lie group associated with $\g$, then a left-invariant coKähler structure on $G$ determines a coKähler structure on the Lie algebra $\g$. Conversely, a coKähler structure on $\g$ extends to a left-invariant coKähler structure on the corresponding simply connected Lie group $G$.
	
	For coKähler Lie algebras, we have the following relation with Kähler Lie algebras, which shows that coKähler structures are the analogs of Kähler structures in odd dimensions.
	
	\begin{teo} \label{F-V} \cite{FV}
		CoKähler Lie algebras of dimension $2n+1$ are in one-to-one correspondence with $2n$-dimensional Kähler Lie algebras endowed with a derivation $D$ which is skew-symmetric and commutes with the complex structure.
	\end{teo}
	
	Under this correspondence, any Lie algebra $\g$ with a coKähler structure $(J,\Z,\alpha,\prodint)$ can be identified with a Kähler Lie algebra $(\h,J,\prodint)$ together with a derivation $D$, where the Hermitian structure on $\h$ is defined by the restriction of $J$ and $\prodint$ to $\h$. Note that there is an abuse of notation denoting by $J$ both the endomorphism on $\g$ and the complex structure on $\h$. Under this identification,
	\[
	\g=\R \xi \ltimes_D \h.
	\]
	Hence the triple $(\xi,\h,D)$ with the Hermitian structure $(\prodint, J)$ on $\h$ encodes the coKähler structure on $\g$, and we will use this notation throughout the paper.

\subsection{Flat Lie algebras}
Let $(\g,\prodint)$ be a metric Lie algebra. The Levi-Civita connection associated to the metric can be computed using the Koszul formula
\[2\<\nabla_XY,Z\>=\<[X,Y],Z\>-\<[Y,Z],X\>+\<[Z,X],Y\>,\]
for any $X,Y,Z\in \g$.
	
	The {\em Riemann curvature tensor} $R$ is given by
	\[R(X,Y)Z=\nabla_X \nabla_YZ-\nabla_Y \nabla_X Z- \nabla_{[X,Y]}Z,\]
	for any $X,Y,Z\in \g$. 
	A metric Lie algebra $(\g,\prodint)$ is said to be {\em flat} if $R\equiv0.$ 
	We recall a characterization of flat Lie algebras originally proved by Milnor \cite{Mi} and then improved by Barberis, Dotti  and Fino \cite{BDF}.
	
	\begin{prop}\cite{BDF}
		Let $(\g,\prodint)$ be a flat Lie algebra. Then $\g$ decomposes orthogonally as 
		\[\g=\mathfrak{z}(\g)\oplus \h \oplus \g^{1},\]
		where $\mathfrak{z}(\g)$ is the center of $\g$, $\h$ is an abelian Lie subalgebra, the commutator ideal $\g^{1}$ is abelian and the following conditions are satisfied: 
        \begin{itemize}
            \item[\ri] $\mathrm{ad}:\h \to \mathfrak{so}(\g^{1})$ is injective and $\g^{1}$ is even dimensional;
            \item[\rii] $\mathrm{ad}_X=\nabla_X$ for any $X\in \mathfrak{z}(\g)\oplus \h.$ 
		\end{itemize}
		In particular, $\mathfrak{g}$ is isomorphic to a Lie subalgebra of $\mathbb{R}^{s} \times \mathfrak{e}(\mathfrak{g}^{1})$, where $\mathfrak{e}(\mathfrak{g}^{1}) = \mathfrak{so}(\mathfrak{g}^{1}) \ltimes \mathfrak{g}^{1}$.
	\end{prop}

	\begin{obs}
		As a consequence, any flat Lie algebra is solvable and unimodular.
	\end{obs}

	\subsection{Almost abelian Lie algebras}
	We recall that a Lie group $G$ is said to be {\em almost abelian} if its Lie algebra $\g$ has an abelian ideal of codimension one. Such a Lie algebra will also be called almost abelian, and can be written as
	\[
	\g=\R x \ltimes_{\mathrm{ad}_x} \mathfrak{u},
	\]
	where $\mathfrak{u}$ is an abelian ideal of $\g$ and $x\notin\u$. Accordingly, the Lie group $G$ is a semidirect product $G=\R \ltimes_{\phi} \R^{d}$ for some $d\in \mathbb{N}$, where the action is given by $\phi(t)=e^{t \mathrm{ad}_x}$. We will usually write $M:=\mathrm{ad}_x|_{\u}$. Clearly, it is $2$-step solvable. Moreover, if $M$ is nilpotent (respectively, $\tr(M)=0$), then $\g$ is nilpotent (respectively, unimodular).
	
	Regarding the isomorphism classes of almost abelian Lie algebras, it is well known that:
	\begin{lema}\label{isomorf casi abel} \cite{casi-abel}
		Two almost abelian Lie algebras $\g_1=\R x_1 \ltimes_{M_1} \mathfrak{u}_1$ and $\g_2=\R x_2 \ltimes_{M_2} \mathfrak{u}_2$ are isomorphic if and only if there exists $c\neq 0$ such that $M_1$ and $cM_2$ are conjugate.
	\end{lema}
	
	In general, it is not easy to determine whether a given Lie group $G$ admits a lattice. A well-known restriction is that, if this is the case, then $G$ must be unimodular \cite{Mi}; equivalently, when $G$ is connected, $\mathrm{tr}(\mathrm{ad}_y)=0$ for every $y\in \g$. An important feature of almost abelian Lie groups is that they admit a concrete lattice criterion:
	\begin{prop}\label{retic casi abel}\cite{Bock}
		Let $G=\R \ltimes_{\phi}\R^{n}$ be a unimodular almost abelian Lie group. Then $G$ admits a lattice if and only if there exists a $t_0\neq0$ such that $\phi(t_0)$ can be conjugated to an integer matrix in $SL(n,\mathbb Z)$.
		
		\noindent In this case, a lattice is given by $\Gamma=t_0\mathbb{Z}\ltimes P^{-1}\mathbb{Z}^n$, where $P\phi(t_0)P^{-1}$ is an integer matrix.
	\end{prop}

	\section{On flat coKähler Lie algebras}
	
	It was proved by Fino and Vezzoni that every unimodular Lie algebra with a coKähler structure must be flat, and hence solvable.
	
	\begin{prop}\label{cokahler unimod}\cite{FV}
		If a Lie algebra $\mathfrak{g}$ admits a coKähler structure and is unimodular, then it is necessarily flat and solvable.
	\end{prop}
	
	This naturally leads to the question of whether the converse of Proposition \ref{cokahler unimod} is true. 
	Since coKähler Lie algebras are in correspondence with Kähler Lie algebras (see Theorem \ref{F-V}), it is natural first to ask whether every even-dimensional flat Lie algebra admits a Kähler structure. This was answered affirmatively by Barberis, Dotti and Fino as a consequence of the following characterization of flat Lie algebras.
	
	\begin{prop}  \label{Flat BDF} \cite{BDF}
		Let $(\mathfrak{g},\langle\cdot,\cdot\rangle)$ be a flat Lie algebra with 
		$\dim(\mathfrak{g}^{1}) = 2m$ and $\dim(\mathfrak{z}(\mathfrak{g})) = s$. 
		Then there exists a matrix $\theta = (\theta_{\beta}^{\alpha}) \in M(k,m;k)$ \footnote{Here $M(k,m;k)$ denotes the set of real $k\times m$ matrices of rank $k$. We write $\theta_\beta=(\theta_\beta^1,\dots,\theta_\beta^k)\in\mathbb R^k$ for the $\beta$-th column of $\theta$.}
		such that $\theta_{\beta} \neq 0$ for each $1 \leq \beta \leq m$, and $\mathfrak{g}$
		admits an orthogonal decomposition
		\[
		\mathfrak{g} \cong \mathbb{R}^{s} \times \bigl(\mathbb{R}^{k} \ltimes_{\rho_{\theta}} \mathbb{R}^{2m}\bigr),
		\]
		where $\{e_{1},\dots,e_{k},f_{1},\dots,f_{2m}\}$ is an orthonormal basis of
		$\mathbb{R}^{k} \ltimes_{\rho_{\theta}} \mathbb{R}^{2m}$, and $T \in \mathbb{R}^{k}$
		acts on $\mathbb{R}^{2m}$ as follows:
		\[
		\rho_{\theta}(T) =
		\begin{pmatrix}
			0 & -\langle T, \theta_1 \rangle & & & \\
			\langle T, \theta_1 \rangle & 0 & & & \\
			& & \ddots & & \\
			& & & 0 & -\langle T, \theta_m \rangle \\
			& & & \langle T, \theta_m \rangle & 0
		\end{pmatrix},
		\]
		where $\langle\cdot,\cdot\rangle$ denotes the Euclidean inner product on $\mathbb{R}^{k}$.
	\end{prop}
	\begin{coro} \label{kahler flat}\cite{BDF}
		Every flat Lie algebra of even dimension is Kähler.
	\end{coro}
	
	\begin{proof}
		Let $\mathfrak{g} = \mathbb{R}^{s} \times (\mathbb{R}^{k} \ltimes_{\rho_{\theta}} \mathbb{R}^{2m})$
		be a Lie algebra as in Proposition~\ref{Flat BDF} (with $s+k$ even).
		Fix an orthonormal basis
		\[
		\mathcal{B}=\{e_{1},\dots,e_{s},e_{s+1},\dots,e_{s+k},f_{1},\dots,f_{2m}\}
		\]
		of $\mathfrak{g}$.
		
		Let $J$ be the orthogonal endomorphism leaving $\mathbb{R}^{s}\times\mathbb{R}^{k}$ invariant
		and such that $J^{2}=-\mathrm{Id}$, and
		\[
		Jf_{2i-1}=f_{2i}, \qquad i=1,\dots,m.
		\]
		With respect to the basis $\mathcal{B}$, the endomorphism $J$ is given by
		\[
		J=
		\left(
		\begin{array}{c|ccc}
			B &  & 0 \\ \hline
			0 & \begin{array}{cc|}
				0 & -1  \\
				1 & 0   \\\hline
			\end{array}\\
			&  &\ddots \\
			& &  &\begin{array}{|cc}
				\hline
				0 & -1 \\
				1 & 0
			\end{array}
		\end{array}
		\right),
		\]
		where $B$ is an $(s+k)\times(s+k)$ block satisfying $B^{2}=-\mathrm{Id}$.
		
		Recall that, with respect to the basis $\mathcal{B}$,
		\[
		\rho_{\theta}(T)=
		\left(
		\begin{array}{c|c}
			0_{s+k} & \\ \hline
			& \begin{array}{cccc}
				0 & -\langle T, \theta_1 \rangle & &  \\
				\langle T, \theta_1 \rangle & 0 & &  \\
				& & \ddots &  \\
				& &  0 & -\langle T, \theta_m \rangle \\
				& &  \langle T, \theta_m \rangle & 0
			\end{array}
		\end{array}
		\right).
		\]
		Observe that $\rho_{\theta}(T)J = J\rho_{\theta}(T)$ for all $T \in \mathbb{R}^{k}$.
		
		We now show that $J$ is integrable.
		If $X,Y \in \mathbb{R}^{s}$, $\mathbb{R}^{k}$, or $\mathbb{R}^{2m}$, or if
		$X \in \mathbb{R}^{s}$ and $Y \in \mathbb{R}^{k}$, then
		\[
		N_{J}(X,Y)=0.
		\]
		Therefore, it remains to check $N_{J}(e_{j},f_{2i-1})$ and $N_{J}(e_{j},f_{2i})$.
		We compute
		\begin{align*}
			N_{J}(e_{j},f_{2i-1})
			&=[Je_{j},Jf_{2i-1}]
			-J[Je_{j},f_{2i-1}]
			-J[e_{j},Jf_{2i-1}]
			-[e_{j},f_{2i-1}] \\
			([X,Y]=\rho_{\theta}(X)Y)\ \Rightarrow\ &=
			\rho_{\theta}(Je_{j})Jf_{2i-1}
			-J\rho_{\theta}(Je_{j})f_{2i-1}
			-J\rho_{\theta}(e_{j})f_{2i}
			-\rho_{\theta}(e_{j})(-Jf_{2i}) \\
			(\rho_{\theta}(T)J=J\rho_{\theta}(T))\ \Rightarrow\ &=0.
		\end{align*}
		Similarly, one checks that $N_{J}(e_{j},f_{2i})=0$, hence $J$ is integrable.
		
		We now show that $\nabla J=0$.
		Recall that $\nabla_{X}=\rho_{\theta}(X)$ for all $X\in\mathbb{R}^{k}$, and
		$\nabla_{X}=0$ for all $X\in\mathbb{R}^{s}$.
		Thus,
		\[
		\nabla J(X,Y)=0,
		\quad \text{for } X,Y\in \mathbb{R}^{k},\mathbb{R}^{s},\mathbb{R}^{2m},
		\ \text{and for } X\in\mathbb{R}^{s},\ Y\in\mathbb{R}^{k}.
		\]
		If $X\in\mathbb{R}^{k}$ and $Y\in\mathbb{R}^{2m}$, then
		\begin{align*}
			\nabla J(X,Y)
			&=\nabla_{X}(JY)-J\nabla_{X}Y \\
			&=\rho_{\theta}(X)JY-J\rho_{\theta}(X)Y \\
			&=0.
		\end{align*}
		Therefore, $(\mathfrak{g},J,g)$ is a flat Kähler Lie algebra.
	\end{proof}
	
	We are now ready to prove the next theorem, which gives not only the converse of Proposition \ref{F-V} but also the precise analogue of the even-dimensional Kähler case proved in Corollary \ref{kahler flat}.

	\begin{teo} \label{thm:principal}
		Let $(\mathfrak{g},\langle\cdot,\cdot\rangle)$ be a flat Lie algebra of odd dimension.
		Then $\mathfrak{g}$ admits a coKähler structure.
	\end{teo}
	
	\begin{proof}
		Since $(\mathfrak{g},\langle\cdot,\cdot\rangle)$ is flat, by Proposition~\ref{Flat BDF} we have
		\[
		\mathfrak{g}\cong \mathbb{R}^{s}\times(\mathbb{R}^{k}\ltimes_{\rho_{\theta}}\mathbb{R}^{2m}),
		\]
		where $s=\dim(\mathfrak{z}(\mathfrak{g}))$ and $2m=\dim(\mathfrak{g}^{1})$.
		Therefore, $s+k$ is odd.
		
		Fix an orthonormal basis
		\[
		\{e_{1},\dots,e_{s},e_{s+1},\dots,e_{s+k},f_{1},\dots,f_{2m}\}
		\]
		of $\mathfrak{g}$.
		The only non-vanishing brackets are given by
		\[
		[X,Y]=\rho_{\theta}(X)Y,
		\qquad X\in\mathbb{R}^{k},\ Y\in\mathbb{R}^{2m}.
		\]
		Moreover, the family $\{\rho_{\theta}(e_i)\}_{i=s+1,\dots,s+k}$ consists of skew-symmetric
		$2m\times2m$ matrices, and we denote $A_i:=\rho_{\theta}(e_i)$, then it can be written by
		\[
		A_i=
		\left(
		\begin{array}{ccccc}
			0 & -a^{i}_{1} & & & \\
			a^{i}_{1} & 0 & & & \\
			& & \ddots  & & \\
			& & & 0 & -a^{i}_{m} \\
			& & & a^{i}_{m} & 0
		\end{array}
		\right), \text{for $s+1\leq i\leq s+k$ and for some $a^i_j\in\R$}.
		\]
		\noindent Fix $j\in\{1,\dots,k\}$,
		then $\mathfrak{g}$ decomposes as
		\[
		\mathfrak{g}\cong \mathbb{R}e_{s+j}\ltimes_{D}
		\bigl(\mathbb{R}^{s}\times(\mathbb{R}^{k-1}\ltimes_{\rho_{\theta}}\mathbb{R}^{2m})\bigr),
		\]
		where
		\[
		D=
		\left(
		\begin{array}{c|c|c}
			0_{s} & \\ \hline
			& 0_{k-1} \\ \hline
			& & A_{i+j}
		\end{array}
		\right).
		\]
		
		\noindent Define next an even-dimensional Lie algebra $\h$ by 
		\[
		\mathfrak{h}=\mathbb{R}^{s}\times(\mathbb{R}^{k-1}\ltimes_{\rho_{\theta}}\mathbb{R}^{2m}).
		\]
		This Lie algebra $\h$ with the restriction of the inner product on $\g$ is a flat Lie algebra. Hence, it follows from Corollary~\ref{kahler flat} that $\h$ admits a Kähler structure where its complex structure can be written as
		\[
		J=
		\left(
		\begin{array}{c|c}
			C & 0 \\ \hline
			0 & \begin{array}{ccccc}
				0 & -1 & & \\
				1 & 0 & & \\
				& & \ddots & \\
				& &  & 0 & -1 \\
				& &  & 1 & 0
			\end{array}
		\end{array}
		\right),
		\]
		where $C$ is an $(s+k-1)\times(s+k-1)$ matrix satisfying $C^{2}=-\mathrm{Id}$. Thus,
		\[
		\mathfrak{g}\cong \mathbb{R}e_{s+j}\ltimes_{D}\mathfrak{h},
		\]
		where $\mathfrak{h}$ is Kähler and $D$ is a skew-symmetric derivation commuting with $J$.
		Hence, by Theorem~\ref{F-V}, the Lie algebra $\mathfrak{g}$ admits a coKähler structure. Moreover, the coK\"ahler structure is determined by $(\xi=e_{s+j}, \h,D)$ with the Hermitian structure on $\h$ given by the natural restriction of $J$ and $\prodint$ to $\h$.
	\end{proof}
	
	\section{On almost abelian coK\"ahler Lie algebras}
	
	In this section we study coKähler structures on a special family of Lie algebras. As mentioned in the preliminaries, almost abelian Lie algebras form a natural class for the study of special geometric structures due to the simple description of their Lie bracket in terms of a single matrix.
	
	We begin by recalling the following characterization of almost abelian Lie algebras admitting a Kähler structure, which appears in \cite{Lauret-Will} (see also \cite{AO,ABDGH}).
	
	\begin{prop} \label{kahler casi abel} 
		A $2n$-dimensional almost abelian Lie algebra $\g= \R x \ltimes_{A} \mathfrak{a}$ admits a Kähler structure if and only if $A=\left(
		\begin{array}{c|c}
			a&0\cdots 0 \\ \hline
			0& \\
			\vdots & \tilde{A} \\0& \\
		\end{array}
		\right)$, where $a\geq 0 $ and $\tilde{A}\in \mathfrak{u}(n-1)$ in some basis of $\mathfrak a$.
	\end{prop}
	
	We now characterize almost abelian Lie algebras admitting a coKähler structure. 
	Let $\g=\R e_0 \ltimes_{M} \u$ be an almost abelian Lie algebra of dimension $2n+1$. Assume that $\g$ admits a coKähler structure, namely $(\g,\Z,J,\alpha,\prodint)$. Equivalently, by Theorem \ref{F-V}, $\g$ decomposes as $\g=\R \Z \ltimes_{D} \h$ where $\h$ is a Kähler Lie algebra of dimension $2n$, and $D$ is a skew-symmetric derivation of $\h$ which commutes with the complex structure $J|_{\h}$.

Without loss of generality, we may assume that $|e_0|=1$ and $e_0 \perp \u$. Then we have two cases:
	
$\ri$ $\h=\u$: in this case $\h$ is abelian, hence it admits a Kähler structure. Since $\xi\perp\h$, we can take $\Z=e_0$. Thus $M=D$ must be skew-symmetric and commute with $J$. In particular, $\g$ is unimodular, hence $(\g,\prodint)$ is flat, and $\mathrm{ker}(M)$ may be trivial.
	
	$\rii$  $\h \neq \u$: in this case, we consider $W=\h \cap \u$, which is an abelian ideal of $\g$ of codimension $2$, that is, $\mathrm{dim}(W)=2n-1$. We denote by $\{w_1,\dots,w_{2n-1}\}$, an orthonormal basis of $W$, and we extend it to $\{u_0,w_1,\dots,w_{2n-1}\}$, and $\{h_0,w_1,\dots,w_{2n-1}\}$, orthonormal bases of $\u$ and $\h$ respectively. Then we have that
	\begin{align} 
		h_0&=\lambda_1e_0+\lambda_2 u_0, &\quad u_0&=\mu_1 \Z + \mu_2 h_0,\label{ecuacion1} \\
		\Z&=\delta_1 e_0 +\delta_2 u_0, &\quad e_0&=\sigma_1\Z+\sigma_2 h_0, \label{ecuacion2}
	\end{align}
	for some $\lambda_i$, $\mu_i$, $\delta_i$ and $\sigma_i$ where \[
	\lambda_1\neq 0, \quad \mu_1\neq0, \quad \delta_2 \neq 0, \quad \sigma_2\neq 0 \quad \text{(since $\h \neq \u$)}.
	\]
	Denoting $\lambda:=(\lambda_1,\lambda_2)$, $\delta:=(\delta_1,\delta_2)$, $\mu:=(\mu_1,\mu_2)$, $\sigma:=(\sigma_1,\sigma_2)$, we have that $\lambda \perp \delta$ and $\mu \perp \sigma$, since $h_0 \perp \Z $ and $u_0 \perp e_0$. In particular, $\delta=\pm(-\lambda_2,\lambda_1)$ and $\sigma=\pm(-\mu_2,\mu_1)$.

	\
	
	We can now state the first structural result.
	\begin{prop}\label{cokahler casi abeliana implica kahler casi abeliana}
		Let $\g=\R e_0 \ltimes_{M} \u$ be an almost abelian Lie algebra of dimension $2n+1$ admitting a coKähler structure $(\Z,\h,D)$. 
		Then $\h$ is abelian or an almost abelian Lie algebra.
	\end{prop}
	\begin{proof}
		If $\h=\u$, then $\h$ is abelian. If $\h\neq\u$, let $W$, $h_0$, $u_0$ be as in case  $\rii$ above. We have that $W$ is an abelian ideal of codimension $1$ in $\h$. Moreover, we have that $[h_0,W]\subseteq W$. Indeed, for any $w\in W$, then $[h_0,w]=rh_0+\tilde{w}$ for some $r\in\R$ and $\tilde{w}\in W$ . Using equation \eqref{ecuacion1} we have
		\[
		[h_0,w]=[\lambda_1e_0+\lambda_2u_0,w]=\lambda_1[e_0,w]\in \u.
		\]
		Then, $[h_0,w]=r(\lambda_1e_0+\lambda_2u_0)+\tilde{w}\in \u$ and thus $r\lambda_1=0$. Since $\lambda_1\neq0$, we get $r=0$, and we conclude that $\h$ is almost abelian. 
	\end{proof}

	\begin{obs}\label{IMPORTANTE}
		An important consequence of Proposition \ref{cokahler casi abeliana implica kahler casi abeliana} is that we can now apply Proposition \ref{kahler casi abel}. Note that $Jh_0\in W$: indeed, $Jh_0$ is perpendicular to $h_0$, and $Jh_0\in\h$ since $\h$ is $J$-invariant. Then the matrix $\ad_{h_0}|_W$ can be written in an orthonormal basis $\{Jh_0,x_1,Jx_1,\dots,x_{n-1},Jx_{n-1}\}$ of $W$ as
		\[\ad_{h_0}|_W=\left(
		\begin{array}{c|c}
			a&0\cdots 0 \\ \hline
			0& \\
			\vdots & \tilde{A} \\0& \\
		\end{array}
		\right), \text{with $a\in \R$ and $\tilde{A}\in \mathfrak{u}(n-1)$}.\]
		Note that $[h_0,Jh_0]=aJh_0$. 
	\end{obs}

	Regarding the Lie bracket of the distinguished elements, we can state the following.
	
	\begin{lema} \label{corchetes}
		With the notation above, we have that: \\
		\ri \ $[\Z,h_0]=0,$ and $[\Z,Jh_0]=0,$ \\
		\rii \ $[e_0,u_0]=0.$ In particular, $M$ has non-trivial kernel.
	\end{lema}
	\begin{proof}
		\ri \ If $w\in W$, then we know that $[\Z,w]\in \h$, and on the other hand, using equation \eqref{ecuacion2},
		\[
		[\Z,w]=\delta_1[e_0,w]+\delta_2[u_0,w]=\delta_1[e_0,w]\in \u,
		\]
		and therefore, $[\Z,W]\subseteq W$.  Consider an orthonormal basis $\{h_0,Jh_0,x_1,Jx_1,\dots,x_{n-1},Jx_{n-1}\}$ of $\h$, then the matrix $D=\ad_{\Z}$ is given in this basis by 
		$D=\left(
		\begin{array}{c|c}
			c&0\cdots 0 \\ \hline
			& \\
			v & \delta_1 \ad_{e_0}|_W \\
			& \\
		\end{array}
		\right)$.
		Since $D$ must be skew-symmetric, we get $c=0$ and $v=0$, that is, $[\Z,h_0]=0$. Finally, $[\Z,Jh_0]=0$ because $D$ commutes with $J$.
		
		$\rii$ It follows from $\ri$ and the equations \eqref{ecuacion1} and \eqref{ecuacion2} above that,
		\begin{align*}
			0=[\Z,h_0]&=[\delta_1e_0+\delta_2u_0,\lambda_1e_0+\lambda_2u_0] \\
			&=(\delta_1\lambda_2-\delta_2\lambda_1)[e_0,u_0]\\
			&=\<(\lambda_1,\lambda_2),(-\delta_2,\delta_1)\>[e_0,u_0].
		\end{align*}
		Since $\lambda \perp \delta$, we have that $\lambda$ and $(-\delta_2,\delta_1)$ are collinear, leading to $[e_0,u_0]=0.$
	\end{proof}
	
		\begin{obs}
			Case $\rii$ above ($\h\neq\u$) implies that $\h$ is not abelian, unless $\g$ is abelian. Indeed, if $\h\neq\u$ and $\h$ is abelian, from $[h_0,w]=\lambda_1[e_0,w]+\lambda_2[u_0,w]$ we have that $[e_0,w]=0$ for all $w\in W$. Therefore, the only non-trivial bracket on $\g$ is $[e_0,u_0]$, but it is also zero from Lemma \ref{corchetes}. Since we are considering non-abelian Lie algebras $\g$, then case $\rii$ above ($\h\neq\u$) implies $\h$ is not abelian.
		\end{obs}
	
	\begin{obs}\label{estructura cokahler}
		As a consequence of the lemma above we have that, in the basis \\
		$\{Jh_0,x_1,Jx_1,\dots,x_{n-1},Jx_{n-1}\}$ of $W$, 
		\[D|_W=\delta_1\ad_{e_0}|_W=\left(
		\begin{array}{c|c}
			0&0\cdots0 \\ \hline
			
			0& \\
			\vdots&\tilde{D} \\
			0&\\
		\end{array}
		\right),
		\text{for some $\tilde{D}\in\u(n-1)$}.\]
	\end{obs}
	
	We now have the ingredients needed to characterize the algebraic structure of an almost abelian Lie algebra admitting a coKähler structure.
	
	\begin{teo}
		\label{casi abeliana ida}
		Let $\g=\R e_0 \ltimes_M \u$ be an almost abelian Lie algebra of dimension $2n+1$. If $\g$ admits a coKähler structure $(\Z,J,\alpha,\prodint)$, then either $M\in\u(n)$, or there exists a basis $\{u_0,Jh_0,x_1,Jx_1,\dots,x_{n-1},Jx_{n-1}\}$ of $\u$ as above such that
		\[
		M=\left(
		\begin{array}{c|c|c}
			0&0&0\cdots0 \\ \hline
			0 & m & 0\cdots0 \\ \hline
			0&0& \\
			\vdots&\vdots&\tilde{M} \\
			0&0&\\
		\end{array}
		\right), \quad \text{with $m\in \R$, and $\tilde{M}\in \u(n-1).$}
		\]  
		
	\end{teo} 
	
	\begin{proof}
		Let $\g=\R \Z \ltimes_{D} \h$ be the decomposition of $\g$ associated with the coKähler structure $(\xi,\h,D)$. If $\h$ is abelian, then $M=D$ up to scale, and therefore $M\in \u(n)$. 
		Assume now that $\h$ is not abelian, and let $W,h_0,e_0,$ and $u_0$ be as in \eqref{ecuacion1} and \eqref{ecuacion2}. By Lemma \ref{corchetes}, we know that $[e_0,u_0]=0$. On the other hand, equation \eqref{ecuacion1} gives
		\[
		\lambda_1\mathrm{ad}_{e_0}|_W+ \lambda_2 \mathrm{ad}_{u_0}|_W=\mathrm{ad}_{h_0}|_W \subseteq W.
		\]
		Hence
		\[
		M|_W=\mathrm{ad}_{e_0}|_W=\frac{1}{\lambda_1}\mathrm{ad}_{h_0}|_W \subseteq W,
		\]
		and with respect to the decomposition $\u=u_0 \oplus W$, the matrix of $M$ has the form
		\[
		M=\left(\begin{array}{c|ccc}
			0&&0 \cdots 0\\ \hline
			0&&&\\
			\vdots& &M|_W & \\
			0 &&&\\
		\end{array}
		\right).
		\]
		Consider now the basis $\{Jh_0,x_1,Jx_1,\dots,x_{n-1},Jx_{n-1}\}$ of $W$ given by Observation \ref{IMPORTANTE}. In this basis, $M|_W$ decomposes as
		$$M|_W=\left(
		\begin{array}{c|c}
			m & 0\cdots0 \\ \hline
			0& \\
			\vdots&\tilde{M} \\
			0&\\
		\end{array}
		\right), \text{for some $m\in\R$, and $\tilde{M}\in \u(n-1)$}.$$
	\end{proof}

	\begin{coro} \label{coro importante}
		Let $\g=\R e_0 \ltimes_M \u$ be an almost abelian Lie algebra. Suppose that $\g$ admits a coKähler structure $(\Z,J,\alpha, \prodint)$, that is, $\g=\R \Z \ltimes_D \h$. Then
		$\g$ is unimodular if and only if either $\h$ is abelian or $[h_0,Jh_0]=0.$
		
		Moreover, if $\g$ is not unimodular, then $\g$ is the direct product of a non-unimodular, almost abelian	Kähler Lie algebra $\h$ with $\R$, that is, $D=0$.
	\end{coro}
	
	\begin{proof}
		If $\h$ is abelian, then $M\in\u(n)$ and hence $\g$ is unimodular.
		
		Assume now that $\h$ is not abelian, and let $W,h_0,u_0,$ and $e_0$ be as in \eqref{ecuacion1} and \eqref{ecuacion2}. Using the basis \[\{Jh_0,x_1,Jx_1,\dots,x_{n-1},Jx_{n-1}\}\] of $W$, we have
		
		$$M|_W=\left(
		\begin{array}{c|c}
			m & 0\cdots0 \\ \hline
			0& \\
			\vdots&\tilde{M} \\
			0&\\
		\end{array}
		\right), \; D|_W=\left(
		\begin{array}{c|c}
			0&0\cdots0 \\ \hline
			0& \\
			\vdots&\tilde{D} \\
			0&\\
		\end{array}
		\right), \; A=ad_{h_0}|_W=\left(
		\begin{array}{c|c}
			a&0\cdots 0 \\ \hline
			0& \\
			\vdots & \tilde{A} \\0& \\
		\end{array}
		\right),$$
		where $\tilde M, \tilde D, \tilde A\in\u(n-1)$, $a,m\in\R$.
		It follows from equation \eqref{ecuacion1} that
		$M|_W=\frac{1}{\lambda_1}\ad_{h_0}|_W$, which implies that $m=\frac{a}{\lambda_1}$ and $\lambda_1\tilde{M}=\tilde{A}$. In particular, $\g$ is unimodular if and only if $m=0$, which is equivalent to $a=0$, that is, $[h_0,Jh_0]=0$.
		
		On the other hand, from equation \eqref{ecuacion1} we have $0=\mu_1\ad_\Z|_W+\mu_2\ad_{h_0}|_W$, so $\mu_2a=0$. 
		If $\g$ is not unimodular, then $a\neq0$, hence $\mu_2=0$. This means that $\Z\in\u$, and therefore $\g=\R\Z\times\h$ is the direct product of a non-unimodular Kähler Lie algebra with $\R$.
	\end{proof}
	
	\begin{obs}
Note that by Theorem \ref{casi abeliana ida} and Lemma \ref{corchetes}, when $\h$ is non-abelian, the $\u$-component of any admissible Reeb vector must lie in $\ker(M)$. 
	\end{obs}

	We now prove a converse to Theorem \ref{casi abeliana ida}. More precisely, we show that every matrix $M$ of the above form gives rise to a coKähler structure, and we determine all possible vectors $\Z$ defining such a structure.
	
	\begin{teo} \label{casi abeliana vuelta}
		Let $\g=\R e_0 \ltimes_M \mathfrak{u}$ be an almost abelian Lie algebra of dimension $2n+1$. If $M$ satisfies one of the following conditions:
		\begin{itemize}
			\item[\ri] $M\in \u(n)$, and $\ker(M)$ is trivial.
			\item[\rii] $M=\left(
			\begin{array}{c|c|c}
				0&0&0\cdots0 \\ \hline
				0 & 0 & 0\cdots0 \\ \hline
				0&0& \\
				\vdots&\vdots&\tilde{M} \\
				0&0&\\
			\end{array}
			\right)$, for some basis of $\g$ and $\tilde{M}\in\u(n-1)$.
			
			\item[\riii] $M=\left(
			\begin{array}{c|c|c}
				0&0&0\cdots0 \\ \hline
				0 & m & 0\cdots0 \\ \hline
				0&0& \\
				\vdots&\vdots&\tilde{M} \\
				0&0&\\
			\end{array}
			\right)$, for some basis of $\g$, $m\neq0$ and $\tilde{M}\in\u(n-1)$.
		\end{itemize}
		Then $\g$ admits a coKähler structure $(\Z,\h,D)$.
		
		Moreover, in $\ri$ $\Z=te_0+w$ for any $t\neq0$ and $w\in\u$. In $\rii$ $\Z=\delta_1e_0+\delta_2u_0$ for any $u_0\in\ker(M)$. Finally, in $\riii$ $\Z$ is any unit vector in $\ker(M)$.
	\end{teo}

	\begin{proof}
		
		$\ri$ Define $\h=\u$ and endow $\h$ with the canonical Kähler structure. Consider $\Z=te_0+w$ for any $t\neq0$ and $w\in\u$, then $D=tM$ and by Theorem \ref{F-V} yields a coKähler structure on $\g$.
		Moreover, assume that there exists another coKähler structure for which $\h\neq\u$, thus it is not abelian. Then we may write $\Z=\delta_1 e_0 +\delta_2 u_0$ as in \eqref{ecuacion2}, and Lemma \ref{corchetes} gives $u_0\in\ker(M)$, a contradiction.
		
		$\rii$
		Let $\mathcal{B}=\{u_0,u_1,\dots,u_{2n-1}\}$ be a basis of $\u$ in which $M$ has the stated form. In particular, $u_0$ is an arbitrary element of $\ker(M)$.
		We want to define $\Z$ so that it determines a coKähler structure.
		Consider $\Z=\delta_1e_0+\delta_2u_0$ where $\delta_1^{2}+\delta_2^{2}=1$.  
		
		Now take $h_0=-\delta_2e_0+\delta_1u_0$. Then $\{\Z,h_0,u_1,\dots,u_{2n-1}\}$ is a new basis of $\g$, and we define
		$\h=\R h_0\ltimes W$, where $W$ is generated by $\{u_1,\dots,u_{2n-1}\}$ and the action is given by \[\left(
		\begin{array}{c|c}
			0 & 0\cdots0 \\ \hline
			0& \\
			\vdots&-\delta_2\tilde{M} \\
			0&\\
		\end{array}
		\right).
		\]
		Proposition \ref{kahler casi abel} implies that $\h$ admits a Kähler structure. Hence $\R\Z\ltimes_{D} \h$, with \[D=\left(
		\begin{array}{c|c|c}
			0&0&0\cdots0 \\ \hline
			0 & 0 & 0\cdots0 \\ \hline
			0&0& \\
			\vdots&\vdots&\delta_1\tilde{M} \\
			0&0&\\
		\end{array}
		\right), \quad \text{with respect to the basis $\{h_0,u_1,\dots,u_{2n-1}\}$},\]
		defines a coKähler structure by Theorem \ref{F-V}. Finally, any $u_0\in\ker(M)$ gives a different vector $\Z$. Moreover, Theorem \ref{casi abeliana ida} and Lemma \ref{corchetes} show that the $\u$-component of $\Z$ must lie in $\ker(M)$; therefore, by letting $u_0$ vary in $\ker(M)$,  the family constructed above exhausts all possible choices of $\Z$.

		$\riii$
		In this case, let $\Z=u_0\in\ker(M)$. Then
		$\g=\R\Z\times\h$, where $\h=\R e_0\ltimes W$ and the action is given by \[\left(
		\begin{array}{c|c}
			m & 0\cdots0 \\ \hline
			0& \\
			\vdots&\tilde{M} \\
			0&\\
		\end{array}
		\right),\] 
		for some basis $\{u_1,\dots,u_{2n-1}\}$.
		Proposition \ref{kahler casi abel} implies that $\h$ admits a Kähler structure, and then Theorem \ref{F-V} yields a coKähler structure on $\R\Z\times\h$.
        Moreover, since
		$\g$ is not unimodular ($m\neq0$), then the same argument as in the proof of Corollary \ref{coro importante} implies that every coK\"ahler structure $(\xi,\h,D)$ satisfies that 
        $\Z\in\u$. By Lemma \ref{corchetes} we have $[e_0,\Z]=0$, hence $\Z\in\ker(M)$. Therefore, every possible Reeb vector arises from a unit vector in $\ker(M)$.
	\end{proof}

	Theorems \ref{casi abeliana ida} and \ref{casi abeliana vuelta} completely characterize almost abelian Lie algebras $\g=\R e_0 \ltimes_M \u$ admitting coKähler structures and also determine all possible such structures $(\Z,D,\h)$.
	Using this characterization, we now describe explicitly the correspondence between almost abelian Lie algebras $\g$ admitting coKähler structures and the corresponding almost abelian K\"ahler Lie algebra $\h$.
	
	\begin{obs}
		It was proved in Proposition \ref{cokahler casi abeliana implica kahler casi abeliana} that if an almost abelian Lie algebra $\g$ admits a coKähler structure $(\xi,D,\h)$, then $\h$ must be an almost abelian K\"ahler Lie algebra. Moreover, if $\h$ is not abelian, then $\h=\R h_0\ltimes_A W$ with $A$ given by a multiple of $M|_W$.
		
		On the other hand, any almost abelian K\"ahler Lie algebra $\h=\R h_0\ltimes_A W$ can be extended to an almost abelian coKähler Lie algebra. Indeed, recall that $M,D$ and $A$ can be written as
		$$M|_W=\left(
		\begin{array}{c|c}
			m & 0\cdots0 \\ \hline
			0& \\
			\vdots&\tilde{M} \\
			0&\\
		\end{array}
		\right), \; D|_W=\left(
		\begin{array}{c|c}
			0&0\cdots0 \\ \hline
			0& \\
			\vdots&\tilde{D} \\
			0&\\
		\end{array}
		\right) \; A=\ad_{h_0}|_W=\left(
		\begin{array}{c|c}
			a&0\cdots 0 \\ \hline
			0& \\
			\vdots & \tilde{A} \\0& \\
		\end{array}
		\right),$$ 
		in the basis $\{Jh_0,u_1,\dots,u_{2n-2}\}$, then it follows from equation \eqref{ecuacion1} and equation \eqref{ecuacion2} that
		$$0=\delta_1a, \; a=\delta_2m, \; \tilde{A}=\delta_2\tilde{M}, \; \tilde{D}=\delta_1\tilde{M}.$$    
		Then, an almost abelian K\"ahler Lie algebra $\h=\R h_0\ltimes_A W$ can be extended to an almost abelian coKähler Lie algebra $\g=\R e_0\ltimes_M \u$ where $\u=\R u_0\oplus W$ and $M|_W$ is a multiple of $A$ and $Mu_0=0$.  	We can summarize this relation in Table \ref{tab:mi_tabla}. 	
	\end{obs}
	
	\begin{table}[H]
		\centering
		\begin{tabular}{|c|c|c|}
			\hline
			\textbf{M} & \textbf{D} & \textbf{A} \\ \hline
			$M\in\u(n)$     & $M$     & $0$     \\ \hline
			$\left(
			\begin{array}{c|c|c}
				0&0&0\cdots0 \\ \hline
				0 & 0 & 0\cdots0 \\ \hline
				0&0& \\
				\vdots&\vdots&\tilde{M} \in\u(n-1)\\
				0&0&\\
			\end{array}
			\right)$     & $\left(
			\begin{array}{c|c|c}
				0&0&0\cdots0 \\ \hline
				0 & 0 & 0\cdots0 \\ \hline
				0&0& \\
				\vdots&\vdots&\delta_1\tilde{M} \in\u(n-1)\\
				0&0&\\
			\end{array}
			\right)$     & $\left(
			\begin{array}{c|c}
				0 & 0\cdots0 \\ \hline
				0& \\
				\vdots&\delta_2\tilde{M} \in\u(n-1)\\
				0&\\
			\end{array}
			\right)$     \\ \hline
			$\left(
			\begin{array}{c|c|c}
				0&0&0\cdots0 \\ \hline
				0 & m\neq0 & 0\cdots0 \\ \hline
				0&0& \\
				\vdots&\vdots&\tilde{M} \in\u(n-1)\\
				0&0&\\
			\end{array}
			\right)$     & $0$   & $\left(
			\begin{array}{c|c}
				m\delta_2 & 0\cdots0 \\ \hline
				0& \\
				\vdots&\delta_2\tilde{M} \in\u(n-1)\\
				0&\\
			\end{array}
			\right)$     \\ \hline
		\end{tabular}
		\caption{Correspondence between $M,D,A$, where $A$ in the basis $\{Jh_0,u_1,\dots,u_{2n-2}\}$ of $W$, $M$ is in the basis $\{u_0,Jh_0,u_1,\dots,u_{2n-2}\}$ of $\u$, and $D$ in the basis $\{h_0,Jh_0,u_1,\dots,u_{2n-2}\}$ of $\h$.}
		\label{tab:mi_tabla}
	\end{table}

		Another direct consequence of our characterization of almost abelian Lie algebras admitting a coK\"ahler structure is the following:
	\begin{coro}\label{cokahler es plana o producto direct}
		Let $\g=\R e_0 \ltimes_M \u$ be an almost abelian Lie algebra of dimension $2n+1$. If $\g$ admits a coKähler structure $(\Z,J,\alpha,\prodint)$, then $(\g,\prodint)$ is flat or $\g$ is the direct product of a non-unimodular, almost abelian Kähler Lie algebra $\h$ with $\R$. 
	\end{coro}
	\begin{proof}
		It follows from Theorem \ref{casi abeliana ida} that $M\in\u(n)$, or $M$ has the following form
		\[
		M=\left(
		\begin{array}{c|c|c}
			0&0&0\cdots0 \\ \hline
			0 & m & 0\cdots0 \\ \hline
			0&0& \\
			\vdots&\vdots&\tilde{M} \\
			0&0&\\
		\end{array}
		\right), \quad \text{with $m\in\R$ and $\tilde{M}\in \u(n-1),$}
		\] 
        for some orthonormal basis $\{u_0,Jh_0,x_1,Jx_1,\dots,x_{n-1},Jx_{n-1}\}$ of $\u$.
		If $m=0$, then $\g$ is unimodular, and hence flat by Proposition \ref{cokahler unimod}. On the other hand, if $m\neq0$, then $\g$ is not unimodular, and Corollary \ref{coro importante} shows that $\g$ is the direct product of a non-unimodular almost abelian Kähler Lie algebra $\h$ with $\R$. 
	\end{proof}
	
	
	\begin{obs}\label{obs:splitting-special}
	Corollary \ref{cokahler es plana o producto direct} shows that every non-unimodular almost abelian Lie algebra with a coK\"ahler structure splits as the direct product of a non-unimodular K\"ahler Lie algebra and $\R$. 
	It is therefore natural to ask whether this splitting phenomenon holds for arbitrary non-unimodular coK\"ahler Lie algebras.
	\end{obs}

	\section{Low-dimensional Lie algebras admitting coK\"ahler structures}
	
	In this section we first apply the structural results of Section $4$ to give explicit lists of almost abelian Lie algebras admitting coK\"ahler structures; we give the lists in dimensions $5$ and $7$. 
	
	Then we answer the question in Observation \ref{obs:splitting-special} and show examples outside the almost abelian setting of non-unimodular coK\"ahler Lie algebras which do not split as a product of a K\"ahler Lie algebra with $\R$. 
	To see such phenomena, we use the classification of four-dimensional K\"ahler Lie algebras due to Ovando \cite{Ov} together with the Fino--Vezzoni correspondence to give the complete list of $5$-dimensional Lie algebras admitting coK\"ahler structures. Note that extensions of K\"ahler $4$-dimensional Lie algebras are mentioned in \cite[Theorem 5.6]{calv-per}, but in that result the K\"ahler Lie algebra is not abelian (therefore the almost abelian cases in dimension $5$ are not included); moreover, several Lie algebras appearing there are isomorphic. We therefore reduce the resulting extensions up to Lie algebra isomorphism.
	
	\subsection[Almost abelian Lie algebras with coK\"ahler structures in low dimensions]{Almost abelian Lie algebras with coK\"ahler structures in low dimensions}
	
	Theorems \ref{casi abeliana ida} and \ref{casi abeliana vuelta} completely characterize almost abelian coK\"ahler Lie algebras. Moreover, Corollary \ref{cokahler es plana o producto direct}, together with the extension procedure summarized in Table \ref{tab:mi_tabla}, allows us to classify all almost abelian Lie algebras carrying a coK\"ahler structure in low dimensions. In particular, we exhibit this classification in dimensions $5$ and $7$. 
	We use the notation
	\[R(\lambda) := \left(\begin{array}{cc}
		0  & -\lambda  \\
		\lambda  & 0
	\end{array}\right), \quad 
	c\oplus R(\lambda) :=\left(
	\begin{array}{c|cc}
		c & 0 & 0 \\ \hline
		0 & 0 & -\lambda \\
		0 & \lambda & 0
	\end{array}
	\right),\quad 
	R(\mu)\oplus R(\lambda) :=\left(
	\begin{array}{cc}
		0 & -\mu \\
		 \mu & 0
	\end{array}
	\right)\oplus \left(
	\begin{array}{cc}
	    0 & -\lambda \\
		 \lambda & 0
	\end{array}
	\right).\]
	
	\subsubsection*{Dimension 5}

	By Proposition \ref{kahler casi abel}, every $4$-dimensional almost abelian Kähler Lie algebra is of the form $	\mathbb{R} h_0 \ltimes_{A} \mathbb{R}^3,$
	where $A=a\oplus R(\lambda)$ with a $\ge 0 $ and $\lambda \geq 0$.
	By Lemma \ref{isomorf casi abel}, up to isomorphism there are only three non-isomorphic $4$-dimensional almost abelian Kähler Lie algebras:
	\[
	\mathfrak{rr}_{3,0}= \mathbb{R} h_0 \ltimes_{1\oplus R(0)} \mathbb{R}^3,\qquad
	\mathfrak{r}'_{4,0,\lambda}=\mathbb{R} h_0 \ltimes_{1\oplus R(\lambda)} \mathbb{R}^3,\ \lambda>0,\quad
	\mathfrak{rr'}_{3,0}= \mathbb{R} h_0 \ltimes_{0\oplus R(1)} \mathbb{R}^3.
	\]
	The notation for these $4$-dimensional Lie algebras comes from \cite{ABDO}; see also Table \ref{tab:corchetes4d}. We extend them to $5$-dimensional Lie algebras using the structural results of the previous section, and we reduce them up to Lie algebra isomorphism using Lemma \ref{isomorf casi abel} again to obtain: 
	\begin{prop}
		Any $5$-dimensional almost abelian Lie algebra admitting a coK\"ahler structure is isomorphic to one and only one of the following:
		\begin{itemize}
			\item $\R^5$,
			\item $\g_{1,0}:=\R \Z \ltimes_{R(1)\oplus R(0)} \R^4$,
			\item $\g_{1,\alpha}:=\R \Z \ltimes_{R(1)\oplus R(\alpha)} \R^4,\quad \alpha\in (0,1]$,
			\item $\R \Z \times \mathfrak{rr}_{3,0}$,
			\item $\R\Z \times \mathfrak{r}'_{4,0,\lambda}, \quad \lambda>0$.
		\end{itemize}
	\end{prop} 
	\begin{proof}
	The abelian K\"ahler Lie algebra $\R^4$ determines the following non-isomorphic coK\"ahler extensions:
	\[
	\R^5,\qquad
	\g_{1,0}:=\R \Z \ltimes_{R(1)\oplus R(0)} \R^4,\qquad
	\g_{1,\alpha}:=\R \Z \ltimes_{R(1)\oplus R(\alpha)} \R^4,\quad \alpha\in (0,1].
	\]
	For the non-unimodular algebras $\mathfrak{rr}_{3,0}$ and $\mathfrak{r}'_{4,0,\lambda}$, Corollary \ref{coro importante} implies that the only almost abelian coK\"ahler extensions are the direct products 
	\[
	\R \Z \times \mathfrak{rr}_{3,0}, \qquad  \R\Z \times \mathfrak{r}'_{4,0,\lambda}, \quad \lambda>0.
	\]
	Finally, the extensions of $\mathfrak{rr'}_{3,0}$ are
	\[
	\R \Z \ltimes_{R(0)\oplus R(d)} \mathfrak{rr}'_{3,0}, \qquad d\in\R,
	\]
	and they are isomorphic to $\g_{1,0}$ by Lemma \ref{isomorf casi abel}.
	\end{proof}
	
	\subsubsection*{Dimension 7}
	Combining Proposition \ref{kahler casi abel} and Lemma \ref{isomorf casi abel}, we have that any $6$-dimensional almost abelian Kähler Lie algebra is isomorphic to one and only one of the following Lie algebras:
	 \begin{align*}
		&
		\h_1= \mathbb{R} h_0 \ltimes_{1\oplus R(0)\oplus R(0)} \mathbb{R}^5, \quad
		\h_2= \mathbb{R} h_0 \ltimes_{1\oplus R(\lambda)\oplus R(0)} \mathbb{R}^5,\,\lambda>0,\quad
		\h_3= \mathbb{R} h_0 \ltimes_{1\oplus R(\lambda_1)\oplus R(\lambda_2)} \mathbb{R}^5,\,0<\lambda_2\leq\lambda_1,\\
		&\h_4= \mathbb{R} h_0 \ltimes_{0\oplus R(1)\oplus R(0)} \mathbb{R}^5,\quad
		\h_5= \mathbb{R} h_0 \ltimes_{0\oplus R(1)\oplus R(\alpha)} \mathbb{R}^5,\,\alpha \in (0,1].
	\end{align*}
	
	Extending them to $7$-dimensional Lie algebras using the structural results of the previous section, and reducing them up to Lie algebra isomorphism by Lemma \ref{isomorf casi abel}, we obtain:
	
	\begin{prop}
	Any $7$-dimensional almost abelian Lie algebra admitting a coK\"ahler structure is isomorphic to one and only one of the following:
	\begin{itemize}
		\item $\R^7$,
		\item $\R\Z\ltimes_{R(1)\oplus R(0)\oplus R(0)} \R^6$,
		\item $\R\Z\ltimes_{R(1)\oplus R(\lambda)\oplus R(0)} \R^6,\, 0<\lambda\leq1$,
		\item $\R\Z\ltimes_{R(1)\oplus R(\lambda)\oplus R(\alpha)} \R^6,\,0<\alpha\leq\lambda\leq1$,
		\item $\R\Z \times \h_1,$
		\item $\R\Z \times \h_2,\, \lambda>0$,
		\item$\R\Z \times \h_3,\, 0<\lambda_2\leq\lambda_1$.
	\end{itemize}
	\end{prop} 
	\begin{proof}
	The abelian K\"ahler Lie algebra extends to
	$$
	\R^{7},\quad
	\R\Z\ltimes_{R(1)\oplus R(0)\oplus R(0)} \R^6,\quad
	\R\Z\ltimes_{R(1)\oplus R(\lambda)\oplus R(0)} \R^6,\, 0<\lambda\leq1, \quad
	\R\Z\ltimes_{R(1)\oplus R(\lambda)\oplus R(\alpha)} \R^6,\,0<\alpha\leq\lambda\leq1.
	$$
	For $\h_1,\h_2,$ and $\h_3$, which are non-unimodular, the almost abelian coK\"ahler extensions are the direct products
	$$
	\R\Z \times \h_1,\quad
	\R\Z \times \h_2,\, \lambda>0,\quad
	R\Z \times \h_3,\, 0<\lambda_2\leq\lambda_1.
	$$
	Finally, the extensions of $\h_4$ and $\h_5$ are
	$\R\Z \ltimes_{R(0)\oplus R(d) \oplus R(0)}\h_4,$ and 
	$\R\Z \ltimes_{R(0)\oplus R(d) \oplus R(\alpha d)} \h_5,\, d\in\R,$
	which are isomorphic to $\R\Z\ltimes_{R(1)\oplus R(0)\oplus R(0)} \R^6$ and $\R\Z\ltimes_{R(1)\oplus R(\lambda)\oplus R(0)} \R^6,\, 0<\lambda\leq1$, respectively.
\end{proof}
	\begin{obs}
	The same construction can be used to give the classification of almost abelian coK\"ahler Lie algebras up to isomorphism in each odd dimension.		
	\end{obs}


	\subsection[Lie algebras carrying coK\"ahler structures in dimension 5]{Lie algebras carrying coK\"ahler structures in dimension 5}
	We now turn to all five-dimensional coK\"ahler Lie algebras. The aim is to complete the characterization in \cite{calv-per} by adding the $5$-dimensional Lie algebras coming from the abelian $4$-dimensional one. We then exhibit examples related to Observation \ref{obs:splitting-special}.
	The following table, taken from \cite{ABDO}, lists the real four-dimensional Lie algebras and their non-zero brackets.
	
	\begin{table}[H]
		\centering
		\scalebox{.85}{
			\begin{tabular}{|c|c|}
				\hline
				Lie algebra & Brackets \\
				\hline
				$\R^{4}$ &$[e_i,e_j]=0, \quad \forall \ i,j\in\{1,2,3,4\}$ \\
				\hline
				$\mathfrak{rh}_{3}$ & $[e_1,e_2]=e_3$ \\
				\hline
				$\mathfrak{rr}_{3}$ & $[e_1,e_2]=e_2,\ [e_1,e_3]=e_2+e_3$ \\
				\hline
				$\mathfrak{rr}_{3,\lambda}$ & $[e_1,e_2]=e_2,\ [e_1,e_3]=\lambda e_3,\ \lambda\in[-1,1]$ \\
				\hline
				$\mathfrak{rr}'_{3,\gamma}$ & $[e_1,e_2]=\gamma e_2 - e_3,\ [e_1,e_3]=e_2 + \gamma e_3,\ \gamma\ge 0$ \\
				\hline
				$\mathfrak{r}_{2}\mathfrak{r}_{2}$ & $[e_1,e_2]=e_2,\ [e_3,e_4]=e_4$ \\
				\hline
				$\mathfrak{r}'_{2}$ & $[e_1,e_3]=e_3,\ [e_1,e_4]=e_4,\ [e_2,e_3]=e_4,\ [e_2,e_4]=-e_3$ \\
				\hline
				$\mathfrak{n}_{4}$ & $[e_4,e_1]=e_2, \ [e_4,e_2]=e_3$ \\
				\hline
				$\mathfrak{r}_{4}$ & $[e_4,e_1]=e_1,\ [e_4,e_2]=e_1+e_2,\ [e_4,e_3]=e_2+e_3$ \\
				\hline
				$\mathfrak{r}_{4,\mu}$ & $[e_4,e_1]=e_1,\ [e_4,e_2]=\mu e_2,\ [e_4,e_3]=e_2+\mu e_3,\ \mu\in\mathbb{R}$ \\
				\hline
				$\mathfrak{r}_{4,\alpha,\beta}$ & $[e_4,e_1]=e_1,\ [e_4,e_2]=\alpha e_2,\ [e_4,e_3]=\beta e_3,$\\
				& $\text{with } -1<\alpha\le\beta\le1,\ \alpha\beta\ne0,\ \text{or } -1=\alpha\le\beta\le0$ \\
				\hline
				$\mathfrak{r}'_{4,\gamma,\delta}$ & $[e_4,e_1]=e_1,\ [e_4,e_2]=\gamma e_2 - \delta e_3,\ [e_4,e_3]=\delta e_2+ \gamma e_3, \  \gamma\in\mathbb{R},\ \delta>0$ \\
				\hline
				$\mathfrak{d}_{4}$ & $[e_1,e_2]=e_3,\ [e_4,e_1]=e_1,\ [e_4,e_2]=-e_2$ \\
				\hline
				$\mathfrak{d}_{4,\lambda}$ & $[e_1,e_2]=e_3,\ [e_4,e_3]=e_3, \ [e_4,e_1]=\lambda e_1,\ [e_4,e_2]=(1-\lambda)e_2,\ \lambda\ge\frac{1}{2}$ \\
				\hline
				$\mathfrak{d}'_{4,\delta}$ & $[e_1,e_2]=e_3,\ [e_4,e_1]=\tfrac{\delta}{2}e_1-e_2,\ [e_4,e_3]=\delta e_3,\ [e_4,e_2]=e_1+\tfrac{\delta}{2}e_2,\ \delta\ge0$ \\
				\hline
				$\mathfrak{h}_{4}$ & $[e_1,e_2]=e_3,\ [e_4,e_3]=e_3,  [e_4,e_1]=\tfrac{1}{2}e_1,\ [e_4,e_2]=e_1+\tfrac{1}{2}e_2$ \\
				\hline
		\end{tabular}}
		\caption{Four-dimensional Lie algebras and their brackets.}
		\label{tab:corchetes4d}
	\end{table}

	We now recall from \cite{Ov} the list of non-equivalent four-dimensional Lie algebras admitting a symplectic form $\omega$ and a complex structure $J$ satisfying
		$\omega(JX,JY)=\omega(X,Y), \ \ \forall X,Y \in \g$.
	We always use the basis $\{e_1,e_2,e_3,e_4\}$, $e^{ij}=e^{i}\wedge e^{j}$, and all coefficients are real.
	
	\begin{table}[H]
		\centering
		\renewcommand{\arraystretch}{1.1}
		\setlength{\tabcolsep}{4pt}
		\begin{tabular}{|c|c|c|}
			\hline
			Lie algebra & \centering $J$ & \centering $\omega$\tabularnewline
			\hline
			$\R^{4}$& $Je_1=e_2, \ Je_3=e_4$ & $e^{12}+e^{34}$ \\
			\hline
			$\mathfrak{rh}_{3}$ &
			$Je_{1}=e_{2},\ Je_{3}=e_{4}$ &
			$\begin{array}{l}
				a_{13+24}(e^{13}+e^{24})+a_{14-23}(e^{14}-e^{23})+a_{12}e^{12}, \\ a_{13}^{2}+a_{14}^{2}\neq0
			\end{array}$
			\\
			\hline
			$\mathfrak{rr}_{3,0}$ &
			$Je_{1}=e_{2},\ Je_{3}=e_{4}$ &
			$a_{12}e^{12}+a_{34}e^{34},\ a_{12}a_{34}\neq0$ \\
			\hline
			$\mathfrak{rr}'_{3,0}$ &
			$Je_{1}=e_{4},\ Je_{2}=e_{3}$ &
			$a_{14}e^{14}+a_{23}e^{23},\ a_{14}a_{23}\neq0$ \\
			\hline
			$\mathfrak{r}_{2}\mathfrak{r}_{2}$ &
			$Je_{1}=e_{2},\ Je_{3}=e_{4}$ &
			$a_{12}e^{12}+a_{34}e^{34},\ a_{12}a_{34}\neq0$ \\
			\hline
			$\mathfrak{r}'_{2} $ &
			$\begin{array}{l}
				J_{1}e_{1} = e_{3}, \; J_{1}e_{2} = e_{4} \\[20pt]
				J_{2}e_{1} = -e_{2}, \; J_{2}e_{3} = e_{4}
			\end{array}$  &
			$\begin{array}{l}
				a_{13-24}(e^{13}-e^{24}) + a_{14+23}(e^{14}+e^{23}),\\
				a_{13-24}^{2} + a_{14+23}^{2} \neq 0 \\[10pt]
				a_{13-24}(e^{13}-e^{24}) + a_{14+23}(e^{14}+e^{23}) + a_{12}e^{12},\\
				a_{13-24}^{2} + a_{14+23}^{2} \neq 0
			\end{array}$ \\
			\hline
			$\mathfrak{r}_{4,-1,-1}$ &
			$Je_{4}=e_{1},\ Je_{2}=e_{3}$ &
			$\begin{array}{l}
				a_{12+34}(e^{12}+e^{34})+a_{13-24}(e^{13}-e^{24})+a_{14}e^{14},\\ a_{12+34}^{2}+a_{13-24}^{2}\neq0
			\end{array}$\\
			\hline
			$\mathfrak{r}'_{4,0,\delta}$ &
			$\begin{array}{l}
				J_1e_{4}=e_{1},\ J_1e_{2}=e_{3},\\ J_{2}e_{4}=e_{1},\ J_{2}e_{2}=-e_{3}
			\end{array}$&
			$a_{14}e^{14}+a_{23}e^{23},\ a_{14}a_{23}\neq0$ \\
			\hline
			$\mathfrak{d}_{4,1}$ &
			$Je_{1}=e_{4},\ Je_{2}=e_{3}$ &
			$a_{12-34}(e^{12}-e^{34})+a_{14}\,e^{14},\ a_{12-34}\neq0$ \\
			\hline
			$\mathfrak{d}_{4,2}$ &
			$\begin{array}{l}
				J_1e_{4}=-e_{2},\ J_1e_{1}=e_{3}\\[10pt]
				J_2e_4=-2e_1, \ J_2e_2=e_3
			\end{array}$ &
			$\begin{array}{l}
				a_{14+23}(e^{14}+e^{23})+a_{24}e^{24}, \ a_{14+23}\neq0 \\[10pt]
				a_{14}e^{14}+a_{23}e^{23}, \ a_{14}a_{23}\neq 0
			\end{array}$ \\
			\hline
			$\mathfrak{d}_{4,\frac{1}{2}}$ &
			\begin{tabular}[c]{@{}l@{}}$J_{1}e_{4}=e_{3},\ J_{1}e_{1}=e_{2}$\\ $J_{2}e_{4}=e_{3},\ J_{2}e_{1}=-e_{2}$\end{tabular} &
			$a_{12-34}(e^{12}-e^{34}),\ a_{12-34}\neq0$ \\
			\hline
			$\mathfrak{d}'_{4,\delta}$ &
			\begin{tabular}[c]{@{}l@{}}$J_{1}e_{4}=e_{3},\ J_{1}e_{1}=-e_{2}$\\ $J_{2}e_{4}=-e_{3},\ J_{2}e_{1}=e_{2}$\\ $J_{3}e_{4}=-e_{3},\ J_{3}e_{1}=-e_{2}$\\ $J_{4}e_{4}=e_{3},\ J_{4}e_{1}=e_{2}$\end{tabular} &
			$a_{12-\delta34}(e^{12}-\delta e^{34}),\ a_{12-\delta34}\neq0$ \\
			\hline
		\end{tabular}
		\caption{Complex structures and symplectic forms.}
		\label{tab:complex-symplectic}
	\end{table}

By Theorem \ref{F-V}, every five-dimensional coK\"ahler Lie algebra is obtained from a four-dimensional K\"ahler Lie algebra $\h$ and a skew-symmetric derivation $D$ commuting with its complex structure. 
Thus, starting from the four-dimensional list in Table \ref{tab:complex-symplectic}, one proceeds in two steps: first, keep only those pairs $(J,\omega)$ for which
\[
\langle X,Y\rangle=\omega(-JX,Y)
\]
is positive definite; second, for each remaining K\"ahler Lie algebra, compute the skew-symmetric derivations commuting with $J$. 
The result is summarized in Table \ref{tab:kahler-derivaciones-dim4}.

\begin{table}[H]

\centering
\scriptsize
\begin{adjustbox}{width=\textwidth}
\begin{tabular}{|c|c|c|c|}
\hline
\textbf{$\h$} & \textbf{$J$} & \textbf{$\prodint$} & \textbf{Skew-symmetric derivations commuting with $J$} \\
\hline
$\R^4$ 
& $Je_1=e_2,\ Je_3=e_4$
& $I_4$
& $\left(\begin{array}{rrrr}
0&-a&-b&-c\\
a&0&c&-b\\
b&-c&0&-d\\
c&b&d&0
\end{array}\right)$ \\
\hline
$\mathfrak{rr}'_{3,0}$
& $Je_1=e_4,\ Je_2=e_3$
& $\operatorname{diag}(a,b,b,a),\ a,b>0$
& $\left(\begin{array}{rrrr}
0&0&0&0\\
0&0&d&0\\
0&-d&0&0\\
0&0&0&0
\end{array}\right)$ \\
\hline
$\mathfrak{rr}_{3,0}$
& $Je_1=e_2,\ Je_3=e_4$
& $\operatorname{diag}(a,a,b,b),\ a,b>0$
& $\left(\begin{array}{rrrr}
0&0&0&0\\
0&0&0&0\\
0&0&0&d\\
0&0&-d&0
\end{array}\right)$ \\
\hline
$\mathfrak{r}_2\mathfrak{r}_2$
& $Je_1=e_2,\ Je_3=e_4$
& $\operatorname{diag}(a,a,b,b),\ a,b>0$
& $0$ \\
\hline
$\mathfrak{r}'_{4,0,\delta}$
& $J_1$
& $\operatorname{diag}(-a,b,b,-a),\ a<0,\ b>0$
& $\left(\begin{array}{rrrr}
0&0&0&0\\
0&0&d&0\\
0&-d&0&0\\
0&0&0&0
\end{array}\right)$ \\
\hline
$\mathfrak{r}'_{4,0,\delta}$
& $J_2$
& $\operatorname{diag}(-a,-b,-b,-a),\ a,b<0$
& $\left(\begin{array}{rrrr}
0&0&0&0\\
0&0&d&0\\
0&-d&0&0\\
0&0&0&0
\end{array}\right)$ \\
\hline
$\mathfrak{d}_{4,2}$
& $J_2$
& $\operatorname{diag}(\frac12a,b,b,2a),\ a,b>0$
& $0$ \\
\hline
$\mathfrak{d}_{4,\frac12}$
& $J_1$
& $aI_4,\ a>0$
& $\left(\begin{array}{rrrr}
0&d&0&0\\
-d&0&0&0\\
0&0&0&0\\
0&0&0&0
\end{array}\right)$ \\
\hline
$\mathfrak{d}'_{4,\delta}$
& $J_3$
& $\operatorname{diag}(-a,-a,-\delta a,-\delta a),\ a<0$
& $\left(\begin{array}{rrrr}
0&d&0&0\\
-d&0&0&0\\
0&0&0&0\\
0&0&0&0
\end{array}\right)$ \\
\hline
$\mathfrak{d}'_{4,\delta}$
& $J_4$
& $\operatorname{diag}(a,a,\delta a,\delta a),\ a>0$
& $\left(\begin{array}{rrrr}
0&d&0&0\\
-d&0&0&0\\
0&0&0&0\\
0&0&0&0
\end{array}\right)$ \\
\hline
\end{tabular}
\end{adjustbox}
\caption{Four-dimensional K\"ahler Lie algebras from Table \ref{tab:complex-symplectic} together with the skew-symmetric derivations commuting with the corresponding complex structure.}
\label{tab:kahler-derivaciones-dim4}

\end{table}

Each row of Table \ref{tab:kahler-derivaciones-dim4} gives a five-dimensional coK\"ahler extension
\[
\g=\R e_5\ltimes_D \h.
\]
In \cite[Theorem 5.6]{calv-per} the authors exhibit a similar table; however, the extensions given by the abelian K\"ahler Lie algebra are not considered. Here we also include the extension of $\R^4$, and we also exhibit in Table \ref{tab:kahler-derivaciones-dim4} the K\"ahler metric for all cases. 

\begin{obs}
In \cite[Theorem 5.6]{calv-per} the authors classify five-dimensional Lie algebras admitting coK\"ahler structures up to equivalence of coK\"ahler structures. 
More precisely, two coK\"ahler structures of the form $(\g,J,\Z,\alpha,\prodint)$ are equivalent if there is a Lie algebra isomorphism preserving $(J,\Z,\alpha,\prodint)$. 
However, many of the examples constructed in this way may yield isomorphic Lie algebras, even if they carry non-equivalent coK\"ahler structures. 
\end{obs}

Therefore, we reduce, up to Lie algebra isomorphism, all $5$-dimensional Lie algebras admitting a coK\"ahler structure obtained from Table \ref{tab:kahler-derivaciones-dim4}, and we get the following classification result.

\begin{prop}\label{prop:dim5-cokahler-all}
Let $\g$ be a five-dimensional non-abelian Lie algebra admitting a coK\"ahler structure. Then $\g$ is isomorphic to one of the following Lie algebras:
\[
\begin{array}{ll}
\g_{1,\alpha}=\R e_5\ltimes_{D}\R^4
& [e_5,e_1]=-e_2,\ [e_5,e_2]=e_1,\ [e_5,e_3]=-\alpha e_4,\ [e_5,e_4]=\alpha e_3,\quad \alpha\in[0,1], \\
& D=R(1)\oplus R(\alpha),\\[1mm] 
\g_1= \mathbb{R}e_5\ltimes_D \mathfrak{rr}_{3,0},
& [e_1,e_2]=e_2,\ [e_5,e_3]=-e_4,\ [e_5,e_4]=e_3, \qquad D=R(0)\oplus R(1),\\[1mm]
\g_5=\mathbb{R}e_5 \times \mathfrak{rr}_{3,0},
& [e_1,e_2]=e_2,\\[1mm]
\g_2=\mathbb{R}e_5 \times \mathfrak{r}_2\mathfrak{r}_2,
& [e_1,e_2]=e_2,\ [e_3,e_4]=e_4,\\[1mm]
\g_{6,\delta}=\mathbb{R}e_5 \times \mathfrak{r}'_{4,0,\delta},
& [e_4,e_1]=e_1,\ [e_4,e_2]=-\delta e_3,\ [e_4,e_3]=\delta e_2,\quad \delta>0,\\[1mm]
\g_3=\mathbb{R}e_5 \times \mathfrak{d}_{4,2},
& [e_1,e_2]=e_3,\ [e_4,e_3]=e_3,\ [e_4,e_1]=2e_1,\ [e_4,e_2]=-e_2,\\[1mm]
\g_4=\mathbb{R}e_5 \ltimes_D \mathfrak{d}_{4,\frac12},
& [e_1,e_2]=e_3,\ [e_4,e_3]=e_3,\ [e_4,e_1]=\frac12e_1,\ [e_4,e_2]=\frac12e_2,\
 [e_5,e_1]=-e_2,\ [e_5,e_2]=e_1,\\
 &D=R(1)\oplus R(0),\\[1mm]
\g_7=\mathbb{R}e_5 \times \mathfrak{d}_{4,\frac12},
& [e_1,e_2]=e_3,\ [e_4,e_3]=e_3,\ [e_4,e_1]=\frac12e_1,\ [e_4,e_2]=\frac12e_2,\\[1mm]
\g_{8,\delta}=\mathbb{R}e_5 \times \mathfrak{d}'_{4,\delta},
& [e_1,e_2]=e_3,\ [e_4,e_1]=\frac{\delta}{2}e_1-e_2,\ [e_4,e_3]=\delta e_3,\ [e_4,e_2]=e_1+\frac{\delta}{2}e_2,\quad \delta>0.
\end{array}
\]

In particular, the first line gives the flat extensions arising from $\h=\R^4$ and $\h=\mathfrak{rr}'_{3,0}$ from Table \ref{tab:kahler-derivaciones-dim4}, while the remaining lines give the non-flat extensions.
\end{prop}

\begin{proof}
	Consider first the semidirect product from the first row in Table \ref{tab:kahler-derivaciones-dim4}. Any Lie algebra in this family is isomorphic to $\R e_5\ltimes_{R(1)\oplus R(\alpha)}\R^4$ for $\alpha\in[0,1]$.
	
	Now, $\mathfrak{r}_2\mathfrak{r}_2$ and $\mathfrak{d}_{4,2}$ admit only the trivial derivation, and therefore they determine $\mathbb{R}e_5 \times \mathfrak{r}_2\mathfrak{r}_2$ and $\g_3=\mathbb{R}e_5 \times \mathfrak{d}_{4,2}$, respectively.
	
	Finally, each one-parameter family of derivations in Table \ref{tab:kahler-derivaciones-dim4} determines only two non-isomorphic extensions: the direct product and the one given by $d=1$, which is denoted by $D$. For example, $\mathfrak{rr}_{3,0}$ determines $\mathbb{R}e_5\times \mathfrak{rr}_{3,0}$ and $\mathbb{R}e_5\ltimes_D \mathfrak{rr}_{3,0}$ with $D=R(0)\oplus R(1)$. 
	
	Computing algebraic invariants for all cases, we can distinguish the non-isomorphic Lie algebras. In addition, one checks that $\mathbb{R}e_5 \ltimes_D \mathfrak{r}'_{4,0,\delta}$, with $D=\left(\begin{array}{rrrr}
		0&0&0&0\\
		0&0&d&0\\
		0&-d&0&0\\
		0&0&0&0
	\end{array}\right)$, is isomorphic to $\mathbb{R}e_5\ltimes_D \mathfrak{rr}_{3,0}$,
	with $D=R(0)\oplus R(1)$. Similarly, $\mathbb{R}e_5 \ltimes_D \mathfrak{d}_{4,\frac12}$,
	 with $D=R(1)\oplus R(0)$, is isomorphic to $\mathbb{R}e_5 \ltimes_D \mathfrak{d}'_{4,\delta}$ with $D=R(1)\oplus R(0)$. This gives the final list of non-isomorphic Lie algebras.
\end{proof}

\begin{obs}
For example, $\g_1$ and $\g_4$ are non-unimodular coK\"ahler Lie algebras which do not split as direct products of a K\"ahler Lie algebra and $\R$. 
These examples explain Observation \ref{obs:splitting-special}: the splitting phenomenon from Corollary \ref{cokahler es plana o producto direct} is a feature of the almost abelian case, not of arbitrary coK\"ahler Lie algebras.
\end{obs}

\section*{Acknowledgements}
The authors would like to thank A. Andrada for his comments and suggestions to improve the exposition of this work.

\end{document}